\documentclass[a4paper,12pt,reqno]{amsart}
\usepackage{tikz}
\usetikzlibrary{arrows,shapes}
\usepackage{amsmath}%
\usepackage{amsfonts}%
\usepackage{amssymb}%
\usepackage{graphicx}
\usepackage{float} 
\usepackage{subfigure} 
\usepackage{lineno,hyperref}
\hypersetup{
	colorlinks = true,
	linkcolor = blue,
	anchorcolor = blue,
	citecolor = blue,
	filecolor = blue,
	urlcolor = blue
}
\newcommand{\C}{\mathbb{C}}
\newcommand{\N}{\mathbb{N}}

\newtheorem{theorem}{Theorem}
\newtheorem*{thmA}{Theorem~A}
\newtheorem*{thmB}{Theorem~B}

\theoremstyle{plain}

\newtheorem{definition}{Definition}

\newtheorem{lemma}{Lemma}

\newtheorem{remark}{Remark}

\numberwithin{equation}{section}
\begin{document}
	\title[Oscillating]{The limit set of iterations of entire functions on wandering domains}
	\author[Huang]{Jiaxing Huang}
	\address{School of Mathematics Sciences, Shenzhen University, Guangdong, 518060, P. R. China}
	\email{hjxmath@szu.edu.cn}
	\author[Zheng]{Jian-Hua Zheng}
	\thanks{The first author was supported by the grant (No. 12201420 and 12231013) of NSF of China, and the second author was supported by the grant (No. 11571193) of NSF of China.}
	\address{Department of Mathematical Sciences, Tsinghua University, Beijing, 100084, P. R. China}
	\email{zheng-jh@mail.tsinghua.edu.cn}
	\begin{abstract}
  We first establish any continuum without interiors can be a limit set of iterations of an entire function on an oscillating wandering domain, and hence arise as a component of Julia sets. 
 Recently, Luka Boc Thaler showed that every bounded connected regular open set, whose closure has a connected complement, is an oscillating or an escaping wandering domain of some entire function. A natural question is: What kind of domains can be realized as a periodic domain of some entire function? In this paper, we construct a sequence of entire functions whose invariant Fatou components can be approached to a regular domain.
	\end{abstract}
	\maketitle
	
	\section{Introduction and Main Results}
	Let $f$ be a transcendental entire function from $\mathbb{C}$ to $\C$, 
 and denote by $f^n$, $n\in\mathbb{N}$ the $n$-th iterate of $f$. The Fatou set $\mathcal{F}(f)$ of $f$ is defined to be the set of points $z\in\hat{\mathbb{C}}$ such that $\{f^n\}$
 forms a normal family in an open neighborhood of $z$.  The complement $\mathcal{J}(f)$ of $\mathcal{F}(f)$ is called the Julia set of $f$. An introduction to the properties of these sets can be found in \cite{Ber91, Ber93}.
 A Fatou component $U$ is said to be $p$-periodic if $p$ is the minimal integer $p> 0$ such that $f^p(U)\subseteq U$, and $U_0, \dots, U_{p-1}$ is called a $p$-cycle;
 if $p=1$, then $U$ is called invariant. In addition, $U$ is called pre-periodic if $f^k(U)$ is contained in a periodic Fatou component, for some $k>0$;
 otherwise, we call $U$ a wandering component of $\mathcal{F}(f)$, or a wandering domain. 
 It is known that there are five possible types of periodic Fatou components \cite{Ber93}, namely, attracting domain, parabolic domain, Siegel disk, Herman ring, and Baker domain.

 In \cite{EL87}, the authors gave the first example of a transcendental entire function $f$ with oscillating wandering domains $D$; it is also the first time that the approximation method was introduced to complex dynamics. This method is developed more recently in \cite{Thaler, MRW22, PS23}. 
 In Eremenko and Lyubich's example \cite[Example 1]{EL87}, all limits functions of the family $(f^n)$ in $D$ are distinct constants, and the set of these constants is discrete and infinite. Hence, it is natural to ask if such a limit set is connected. On the other hand, recently, Mart\'i-Peter et al \cite[Theorem 1.2]{PRW22} showed that a planar continuum without interiors is a Julia component of a transcendental meromorphic function.
 
 In this paper, we will show that any continuum without interiors can be a limit set of iterations of an entire function on an oscillating wandering domain, and hence it is also a Julia component of the entire function.

 \begin{theorem}\label{thm:EO1}
    Let $J$ be a continuum with an empty interior, then there exists a transcendental entire function $f$ which has an oscillating wandering domain $\Omega$ such that $J\cup\{\infty\}$ is exactly a limit set of $f^n$ on $\Omega$.
\end{theorem}
\begin{remark}
    Indeed, the continuum $J$ above can be replaced by a simple curve $\gamma:[0,\infty)\to\C$. This is because, for each such a simple curve $\gamma$, one can use a sequence $(X_j)$ of closed half-strips such that $X_j\subset\text{Int}(X_{j-1})$ and $\cap_{j=0}^{\infty}X_j=\gamma$. Replacing Runge's approximation (see Theorem A) with Arakelyan's Theorem (see Theorem B), and mimicking the argument of Theorem \ref{thm:EO1}, one can obtain that $\gamma$ is exactly a limit set of $f^n$ on $\Omega$.
\end{remark}
In the next part, we are interested in the periodic Fatou component of entire functions. In \cite{Thaler}, the author studied the geometry of simply connected wandering domains for entire functions and showed that every bounded connected regular open set, whose closure has a connected complement, is an escaping/oscillating wandering domain of some entire function. 
A natural question is whether such a regular domain can be realized as a periodic Fatou component of entire functions. It is known that the unit disk is an attracting domain of $z^2$, however, in general, it seems impossible. For example, the following drop shape domain (see Figure \ref{Fig1})  is bounded, regular, and connected, and its complement is connected. If such a domain is invariant, then its boundary is also invariant, and the preimage of the angle has infinitely many components in the boundary, this is impossible.
\begin{figure}[H]
	\centering
		\includegraphics[width=0.2\textwidth]{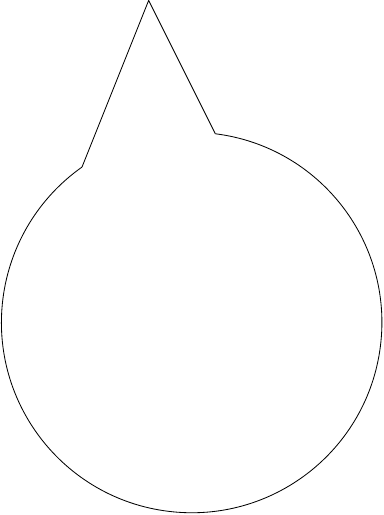}
  \vspace{1em}
		\caption{\small the drop domain  }
		\label{Fig1} 
	\end{figure}

 Thus, we pose the following question:\\
 
 \textbf{Question:} \emph{What kind of domains could be realized as a periodic Fatou component?}\\ 
 
 We cannot completely answer this question, but we get that there exists an entire function with an invariant domain sufficiently closed to a bounded connected regular domain.
  \begin{theorem}\label{thm:F}
    Let $D$ be a bounded simply connected regular open set. Then for any $\delta>0$, there exists an entire function $f$ for which $D^{-\delta}$ is contained in a periodic Fatou component $U$ with $D^{-\delta}\subset U\subset D$, where $D^{-\delta}$ is a subset of $D$ defined by
$$D^{-\delta}:=\{z\in D: \text{dist}(z, \C\setminus D)>\delta\}.$$
     \end{theorem}
     For any $\delta_n>0$ tending to $0$, we can construct a transcendental entire function $f_{\delta_n}$ with an invariant Fatou component $U_{\delta_n}$ such that $D^{-\delta_n}\subset U_{\delta_n}\subset D$.
    If such a sequence $(f_{\delta_n})_n$ of entire functions is convergent to a transcendental entire function $f$, then $D$ must be a periodic domain of $f$. 
    Thus, to answer the above question, it seems sufficient to tackle when such a sequence $(f_{\delta_n})$ converges locally uniformly to an entire function $f$.

In Theorem \ref{thm:F}, the periodic domain is bounded, which is not a Baker domain. Thus, we also consider the case of Baker domains. To state our result, we first introduce some definitions.
\begin{definition}
    Let $U$ be a planar domain in $\C$. 
    We say $U$ is nice if $\infty$ is accessible along a curve $\gamma$ in $U$. 
\end{definition}
\begin{lemma}\label{lem:nice}
Let $U$ be a nice simply connected domain in $\C$. Then
    there exists a holomorphic map $h:U\to \C$ 
    and an absorbing domain $W$ in $U$
    such that 
    \begin{itemize}
        \item $\overline{W}\subset U$
        \item $h(\overline{U})\subset U$, $h^n(\overline{W})=\overline{h^n(W)}\subset h^{n-1}(W)$ for all $n\geq 1$;
        \item $\cap_{n\geq 1}h^n(\overline{W})=\emptyset.$
        \item $\text{dist}(h^n(W), h^{n-1}(\partial W))>d>0$ for all $n\geq 1$ and some fixed constant $d$.
    \end{itemize}
\end{lemma}
\begin{proof}
    Let $\mathbb{H}:=\{z:\Re z>0\}$ be a right half-plane and $\mathbb{H}_1:=\{z: \Re z>1\}$. As $U$ is a nice simply connected domain, it follows from the Riemann mapping theorem that there exists a conformal map $\varphi:\mathbb{H}\to U$ such that $\varphi(\infty)=\infty$. Take $h(z):=\varphi\circ T\circ \varphi^{-1}$ and $W:=\varphi(\mathbb{H}_1)$, where $T(z)=z+1$, then we can obtain the result.
\end{proof}
     \begin{theorem}\label{thm:Baker}
     Let $U$ be a nice and simply connected open set contained in a half-plane such that $\partial U=\partial(\C\setminus \overline{U})$. Then for any $\delta>0$, there exists an entire function $f$ for which $U^{-\delta}$ is contained in an invariant Baker domain $\Omega$ with $U^{-\delta}\subset\Omega\subset U$.
     \end{theorem}

 \section{Preliminaries}
\subsection{Notation}
An open set $U$ is said to be regular if $U=\text{Int}(\overline{U}).$ 
A closed set $A$ is said to Weierstrass if $\hat{\C}\setminus A$ is connected and locally connected at $\infty$, where $\hat{\C}=\C\cup\{\infty\}$.
We denote by $\Delta(a, r)$ an open disk centered at $a$ of radius $r$, by $\overline{U}$ the closure of the set $U$, by $\text{Int}(U)$ the interior of $U$, by $\partial U$ the boundary of $U$, and by $U'$ the derived set of $U$. By $V\Subset U$ we mean that $V$ is compactly contained in $U$. Finally, by $\omega_f(U)$ we mean the set of all points $w$ for which there is some point $z\in U$ whose orbit under $f$ accumulates on $w$.
    
\subsection{Approximation results} We require the following two stronger versions of the well-known Runge's approximation theorem and Arakelyan's approximation theorem.
 \begin{thmA}[\cite{Thaler}, Theorem 4]
    Let $(K_m)_{m\geq 0}$ be a sequence of compact sets pairwise disjoint whose complements $\C\setminus K_m$ are connected. Let $L_k\subset K_k$ be a finite set of points and $f_k:K_k\to\C$ a holomorphic map for every $k$. For every $\epsilon>0$, there exists an entire function $f$ satisfying:
    \begin{itemize}
        \item $|f_k-f|_{K_k}<\epsilon$,
        \item $f(x)=f_k(x)$ for all $x\in L_k$,
        \item $f'(x)=f_k'(x)$ for all $x\in L_k$
    \end{itemize} for every $k$.
\end{thmA}
Similarly, one can also get the following Arakelyan's approximation theorem.
\begin{thmB}[Arakelyan's theorem]
    Let $A\subset\C$ be a closed Weierstrass set, and $L\subset A$ be a finite set of points. Suppose that $g:A\to\C$ is a continuous function that is holomorphic on Int$(A)$. Then for every $\epsilon>0$, there exists an entire function $f$ such that  \begin{itemize}
        \item $|f-g|_A<\epsilon$, 
        \item $f(x)=g(x)$ for all $x\in L$,
        \item $f'(x)=g'(x)$ for all $x\in L$
    \end{itemize}
\end{thmB}
We also require the following elementary result on approximation.
\begin{lemma}[Corollary 2.7, \cite{MRW22}]\label{lem:MRW}
Let $U\subset\C$ be open, and $g:U\to\C$ be holomorphic. Suppose that $G\subset U$ is open and such that $g^n$ is well-defined and univalent on $G$ for some $n\geq 1$.
Then for every closed set $K\subset G$ and every $\epsilon>0$, there is a $\delta>0$ with the following property. For every holomorphic function $f:U\to\C$ with
$$|f-g|_U\leq\delta,$$ the function $f^n$ is defined and injective on $K$, with
$$|f^k-g^k|_K\leq\epsilon\quad \text{for}\ 1\leq k\leq n.$$
\end{lemma}
	\section{Proof of Theorem \ref{thm:EO1}}
 We use a similar vein as the argument in \cite{Thaler}.

It is clear that for any continuum set $J$ without interiors, there exists a sequence, say $(B_{m,n})_{0\leq m, 1\leq n\leq 2^m}$, of disjoint open disks whose Hausdorff limit is $J$.

Without loss of generality, we assume that $J\subset \Delta(0, \frac{1}{3})$, 
Then there exists a sequence $(X_j)_{j=0}^{\infty}$ of compact sets such that 
\begin{itemize}
    \item $X_0\subset\overline{\Delta(0, \frac{1}{3})}$,
    \item $X_j\subset\text{Int}(X_{j-1})$ for all $j\geq 1$,
    \item $\cap_{j=0}^{\infty}X_j=J$.
\end{itemize}
Thus, we can choose the sequence $B_{m,n}:=\Delta(a_{m,n}, b_m)$ of open disks such that 
\begin{itemize}
    \item $B_{m,n}\subset\text{Int}(X_m)\setminus X_{m+1}$, for $0\leq m, 1\leq n\leq 2^m$
    \item $B_{m,n}\cap B_{m,n_1}=\emptyset$ for $n\neq n_1$,
    \item $(B_{m,n})'=J$, that is, $J$ is the Hausdorff limit of the sets $B_{m,n}$.
\end{itemize}
Now, we choose a bounded connected regular open set $\Omega\Subset\Delta(\frac{2}{3},\frac{1}{9})$ whose closure has a connected complement,  such that there exists a sequence $(\Omega_n)_{n\geq 0}$ of compact neighborhoods of $\overline{\Omega}$ satisfying
 \begin{itemize}
     \item[(i).] $\C\setminus\Omega_n$ is connected for all $n\in\N$,
     \item[(ii).] $\Omega_{n+1}\subseteq\text{Int}(\Omega_n)$ for all $n\in\N$,
     \item[(iii).] $\overline{\Omega}=\cap_{n\in\N}\Omega_n$.
 \end{itemize}
  We also choose a sequence $(x_n)\in\C\setminus\overline{\Omega}$ such that 
 \begin{itemize}
     \item[(1).] $x_n\in\text{Int}(\Omega_{n-1})\setminus\Omega_n$ for all $n\geq 1$,
     \item[(2).] $(x_n)'=\partial\Omega$
 \end{itemize}
 Set 
 $\Delta_m:=\Delta(4m,1)$, $\Delta_{m,n}=\Delta(b_{m,n},1),$ for $m\geq 1$ and $1\leq n\leq 2^m,$  where $b_{m,n}=4m+2n\sqrt{-1}.$ 
 
 Let $(C_k)_{k\geq 1}$ be a sequence of bounded simply connected domains such that 
 $$\overline{\Delta(0,1)}\subset C_1,\ \ C_k\Subset C_{k+1},\ \ \bigcup_{k\geq 1}C_k=\C$$ and 
 $$\Delta_{k,n}\cup\Delta_k\subset C_{k+1}\setminus\overline{C_{k}}, 1\leq k,  1\leq n\leq 2^k.$$
Now, we will construct our function $f$ by taking a limit of an inductively constructed sequence of entire functions as the following.
 \begin{lemma}\label{lem:1}
     Let $(\Omega_n)_{n\geq 0}$, $(x_n)_{n\geq 1}$, and $(B_{m,n})_{0\leq m, 1\leq n\leq 2^m}$ be as above. Let $N_m=2^{m+1}-2$, $m\geq 0$. Then there exists a sequence $(f_k)_{k\geq 1}$ of entire functions and a sequence $(x_n^j)_{n, j\geq1}$ of points such that for all $k\geq 1$:
     \begin{enumerate}
         \item $|f_{k+1}-f_k|_{\overline{C_k}}\leq 2^{-k}$,
         \item $f^{N_{m}+2n-1}_k(\Omega_k)\subset B_{m,n}$, for  $0\leq m\leq k-1$ and $1\leq n\leq 2^{m}$.
         \item $f^{N_m}_k(\Omega_k)\subset \Delta_m$, for $1\leq m\leq k$,
         \item $f_k^j|_{\Omega_k}$ and $f_k^j|_{B_{m,n}}$ are univalent for all $1\leq j\leq N_{k-1}$, $0\leq m\leq k-1$ and $1\leq n\leq 2^m$.
         \item $f_k^j(x_n)=x_n^j$ for all $0\leq j\leq N_{k-1}$ and $1\leq n\leq k$, (here $x_n^0=x_n$).
         \item $f_k^{N_{n}}(x_n)=1$ for all $1\leq n\leq k$,
         \item $f_k(1)=1$ and $f_k'(1)=\frac{1}{2}.$
     \end{enumerate}
 \end{lemma}
 A simple flowchart can be shown as the following:
 $$\Omega_1\to B_{0,1}\to\Delta_1;$$
 $$\Omega_2\to B_{0,1}\to\Delta_1\to B_{1,1}\to\Delta_{1,1}\to B_{1,2}\to\Delta_2;$$
$$\cdots$$
$$\Omega_k\to \cdots \to \Delta_m\to B_{m,1}\to\Delta_{m,1}\to B_{m,2}\to\cdots\to B_{m,2^{m}}\to\cdots \to\Delta_k;$$
$$\cdots.$$
It is clear that from $\Omega_k$ to $\Delta_k$, we have to take $N_k:=2(2^k-1)$ steps.
\begin{proof}[Proof of Lemma \ref{lem:1}]
First, we construct an entire function $f_1$ that satisfies:
   \begin{itemize}
   \item[(i).] $f_1^{N_0+1}(\Omega_1)=f_1(\Omega_1)\subset B_{0,1}$,
   \item[(ii).] 
   $f_1^{N_1}(\Omega_1)=f_1^2(\Omega_1)\subset\Delta_1$,
   \item[(iii).] $f_1^j|_{\Omega_1}$ and $f_1^j|_{B_{0,1}}$ are univalent for $1\leq j\leq N_0$,
   \item[(iv).] $f_1^{N_1}(x_1)=f_1(x_1)=1$,
   \item[(v).] $f_1(1)=1$ and $f_1'(1)=\frac{1}{2}$.
   \end{itemize}

Define compact sets $K_1=\{1\}$, $K_2=\{x_1\}$, $K_3=\Omega_1$ and $K_{m,n}=\overline{B_{m,n}}$ for $0\leq m, 1\leq n\leq 2^m$. It is clear that all of them are disjoint and their complements are connected. 

Let $h_1(z)=\frac{1}{2}z+\frac{1}{2}$, $h_2(z)=1$, $h_3$ be a non-constant linear map such that $h_3(K_3)\Subset B_{0,1}$. Let 
$h_{m,n}(z)=\frac{1}{2}(z-a_{m,n})+b_{m,n}$, $0\leq m, 1\leq n\leq 2^m-1$, and $h_{m,2^m}(z)=z-a_{m,2^m}+4(m+1)$ for $m\geq 1$, that is, 
\begin{align}\label{eqn:hmn}
h_{m,n}(B_{m,n})\subset \Delta(b_{m,n}, \frac{b_m}{2})\Subset\Delta_{m,n}, 1\leq m, 1\leq n\leq 2^m-1\end{align} and \begin{align}\label{eqn:hm}
h_{m,2^m}(B_{m,2^m})\subset\Delta(4(m+1), b_m)\Subset \Delta_{m+1}, m\geq 1.\end{align}

By Theorem A, for every $\epsilon_1\in(0,1)$, there exists an entire function $f_1$ such that 
\begin{itemize}
    \item $|f_1-h_j|_{K_j}\leq\frac{\epsilon_1}{2}$ for all $1\leq j\leq 3$,
    \item $|f_1-h_{m,n}|_{K_{m,n}}\leq\frac{\epsilon_1}{2}$ for all $0\leq m, 1\leq n\leq 2^m$.
    \item $f_1(x_1)=1$
    \item $f_1(1)=1$ and $f_1'(1)=\frac{1}{2}.$
\end{itemize}
Clearly, (i) and (ii) hold as long as $\epsilon_1$ is sufficiently small. (iii) follows from the injectivity of $h_3$ and $h_{0,1}$ on $\Omega_1$ and $B_{0,1}$ respectively.
(iv) can be obtained from $f_1^{N_1}(x_1)=f_1(f_1(x_1))=f_1(1)=1$. Clearly, $f_1$ satisfies all conditions of Lemma \ref{lem:1}.
Now, we can define $x_2^j=f_1^j(x_2)$ for $0\leq j\leq N_1$. 

Suppose that we have already obtained entire functions $f_1, \dots, f_k$ and points $x_n^j$, where $0\leq j\leq N_{n-1}$ and $1\leq n\leq k+1$.

Let us define compact set $K_1:=\overline{C_k}$, $K_2:=\{x_{k+1}^{N_k}\}=\{f_k^{N_k}(x_{k+1})\}$ and $K_{k,n}:=\overline{\Delta_{k,n}}, 1\leq n\leq 2^k-1$. By induction, we have 
\begin{itemize}
    \item $f_k^{N_{m}+2n-1}(\Omega_k)\subset B_{m,n}$, for $1\leq m\leq k-1$ and $1\leq n\leq 2^m$
    \item  $f_k^{N_m}(\Omega_k)\subset\Delta_m$ for $1\leq m\leq k$.
\end{itemize}
Thus, compact sets $K_1$, $K_2$, $f_k^{N_k}(\Omega_{k+1})$ and $K_{k,n}$ are pairwise disjoint and have connected complement, hence there exists a compact set $K_3\subset\Delta_k$ such that $f_k^{N_k}(\Omega_{k+1})\subset f_k^{N_k}(\Omega_k)\subset\text{Int}(K_3)$ and $K_j$ and $K_{k,n}$ are pairwise disjoint and have connected complement, for $1\leq j\leq 3$ and $1\leq n\leq 2^k-1$.
Now, we define functions
$$g_1(z):=f_k(z),\ g_2(z):=1,$$ and let $g_3$ be a non-constant linear map that maps $K_3$ into $B_{k,1}$, and $g_{k,n}$ be a non-constant linear map such that \begin{align}\label{eqn:gkn}
g_{k,n}(K_{k,n})\Subset \Delta(a_{k,n+1},\frac{b_k}{2})\subset B_{k,n+1}, 1\leq n\leq 2^k-1.
\end{align}

By Theorem A, for every $\epsilon_{k+1}$, there exists an entire function $f_{k+1}$ such that:
\begin{itemize}
    \item[(i).] $|f_{k+1}-g_j|_{K_j}\leq\epsilon_{k+1}$ for $j=1,2,3$, and $|f_{k+1}-g_{k,n}|_{K_{k,n}}\leq\epsilon_{k+1}$ for $1\leq n\leq 2^k-1$,
    \item[(ii).] $f_{k+1}(x_n^j)=f_k(x_n^j)$ for all $0\leq j\leq N_{n-1}$ and $1\leq n\leq k$, (here $x_n^0:=x_n$,
    \item[(iii).] $f_{k+1}(x_{k+1}^j)=f_k(x_{k+1}^j)$ for all $0\leq j\leq N_k-1$,
    \item[(iv).] $f_{k+1}(x_{k+1}^{N_k})=1$,
    \item[(v).] $f_{k+1}(1)=1$ and $f_{k+1}'(1)=\frac{1}{2}$.
    \end{itemize}
    where $\epsilon_{k+1}\leq 2^{-k}$ can be chosen to be sufficiently small such that:
    \begin{itemize}
        \item[(a).] $f_{k+1}^{N_m+2n-1}(\Omega_{k+1})\subset B_{m,n},$ for all $1\leq m\leq k$ and  $1\leq n\leq 2^m$. 
      \item[(b).] $f_{k+1}^{N_{m}}(\Omega_{k+1})\subset\Delta_{m}$, for all $1\leq m\leq k+1$.
      \item[(c).] $f_{k+1}^j|_{\Omega_{k+1}}$ and $f_{k+1}^j|_{B_{m,n}}$ are univalent for all $1\leq j\leq N_{k}$, $0\leq m\leq k$ and $1\leq n\leq 2^m$.
    \end{itemize}
      Indeed, for (a), we first show that
      for 
      $ 1\leq j\leq 2^k-1$,  
      $$ f_{k+1}(B_{k,j})\subset \Delta_{k,j},\ f_{k+1}(\Delta_{k,j})\subset B_{k,j+1},\  f_{k+1}(B_{k,2^k})\subset\Delta_{k+1}.$$

    Recall that $h_{m,n}(z)=\frac{1}{2}(z-a_{m,n})+b_{m,n}$, $h_{m,2^m}(z)=z-a_{m,2^m}+4(m+1)$ and $B_{m,n}=\Delta(a_{m,n},b_m)$ for $0\leq m$, $1\leq n\leq 2^m$. Then for $1\leq j\leq 2^k-1$, from (\ref{eqn:hmn}), we have
    \begin{align*}
        |f_{k+1}-b_{k,j}|_{B_{k,j}}&\leq\sum_{n=1}^{k}|f_{n+1}-f_n|_{B_{k,j}}+|f_1-h_{k,j}|_{B_{k,j}}+|h_{k,j}-b_{k,j}|_{B_{k,j}}\\
        &\leq \sum_{n=1}^{k+1}\epsilon_n+\frac{b_k}{2}\leq\sum_{n=1}^{\infty}\epsilon_n+\frac{b_k}{2}\leq\frac{1}{2}+\frac{b_k}{2}<1.
    \end{align*} 
   Similarly, from (\ref{eqn:hm}), we have 
    \begin{align*}
        |f_{k+1}-4(k+1)|_{B_{k,2^k}}
        &\leq \sum_{n=1}^{k+1}\epsilon_n+b_k\leq\sum_{n=1}^{\infty}\epsilon_n+b_k<1,
    \end{align*} 
   hence
    $$f_{k+1}(B_{k,2^k})\subset \Delta_{k+1}.$$
     Finally, from (\ref{eqn:gkn}), we have
    \begin{align*}
        |f_{k+1}-a_{k,j+1}|_{\Delta_{k,j}}&\leq|f_{k+1}-g_{k,j}|_{\Delta_{k,j}}+|g_{k,j}-a_{k,j+1}|_{\Delta_{k,j}}\\
        &\leq \epsilon_{k+1}+\frac{b_k}{2}<b_k.
    \end{align*} 
     Observe that $f_k^{N_m}, 1\leq m\leq k$ is well defined and univalent on $\Omega_{k+1}$ by induction, with $f_k^{N_m+2n-1}(\Omega_{k+1})\subset B_{m,n}, 1\leq m\leq k-1, 1\leq n\leq 2^m$, and $f_k^{N_k}(\Omega_{k+1})\subset \text{Int}(K_3).$
     By Lemma \ref{lem:MRW}, we have
    \begin{equation}\label{eqn:k}
    |f_{k+1}^{j}-f_k^{j}|_{\Omega_{k+1}}<\epsilon,\ 1\leq j\leq N_k,
    \end{equation}
    thus for $\epsilon$ sufficiently small, 
    $$f_{k+1}^{N_m+2n-1}(\Omega_{k+1})\subset B_{m,n}, 1\leq m\leq k-1, 1\leq n\leq 2^m$$ and
    $$f_{k+1}^{N_k}(\Omega_{k+1})\subset\text{Int}(K_3).$$
    Therefore,
    $$f_{k+1}^{N_k+2n-1}(\Omega_{k+1})\subset f_{k+1}^{2n-1}(f_{k+1}^{N_k}(\Omega_{k+1}))\subset f_{k+1}^{2n-1}(\text{Int}(K_3)).$$
    For sufficiently small $\epsilon_{k+1}>0$, the relation
    $$f_{k+1}(\text{Int}(K_3))\subset B_{k,1}$$ follows from $|f_{k+1}-g_3|_{K_3}\leq\epsilon_{k+1}.$
    Thus, we have
    $$f_{k+1}^{2n-1}(\text{Int}(K_3))\subset f_{k+1}^{2n-1}(B_{k,1})\subset B_{k,n}, 1\leq n\leq 2^k.$$

For (b), recall that $N_{k+1}=N_k+2^k+1$, then from (a),
    $$f^{N_{k+1}}_{k+1}(\Omega_{k+1})\subset f_{k+1}(f_{k+1}^{N_k+2^k}(\Omega_{k+1}))\subset f_{k+1}(B_{k,2^k})\subset\Delta_{k+1},$$
    and for $1\leq m\leq k$, the result follows by the induction and (\ref{eqn:k}) as long as $\epsilon$ is small enough.

    For (c), the result can be obtained from Lemma \ref{lem:MRW}.

    Finally, we define $x_{k+2}^j:=f_{k+1}^j(x_{k+2})$ for $1\leq j\leq N_{k+1}.$ The entire function $f_{k+1}$ now is constructed and satisfies all the conditions (1)-(7) of Lemma \ref{lem:1}, hence this completes the inductive steps.

    Now, let us go back to show Theorem \ref{thm:EO1}. Let $N_m=2^{m+1}-2$ and let $f$ be the limit of the sequence $(f_k)_{k\geq 1}$ of entire functions given by Lemma \ref{lem:1}. It is obvious that
    \begin{enumerate}
        \item $f^{N_m+n}(\Omega)\subset B_{m,n}$ for $0\leq m,$ and $1\leq n\leq 2^m$,
        \item $f^{N_m}(\Omega)\subset\Delta_m$ for $1\leq m$,
        \item $f^n(\Omega)\subset\cup_{m\geq 0}\Delta_m$,
        \item $f^{N_n}(x_n)=1$ for all $n\geq 1$,
        \item $f^n|_{\Omega}$ is univalent for $n\geq 1$,
        \item $f(1)=1$ and $f'(1)=\frac{1}{2}.$
    \end{enumerate}
    Thus, $\Omega$ is contained in a Fatou set. The pre-images of the attracting fixed point 1 accumulate everywhere on $\partial\Omega$, hence $\Omega$ is a Fatou component which is an oscillating wandering domain from (1) and (2) above.
    Finally, property (1) implies that any point $a\in J$ is a limit point for every orbit originating from $\Omega$, hence $J= \omega_f(\Omega)$.
\end{proof}

 \section{Proof of Theorem \ref{thm:F}}
     Similar to the argument of Theorem \ref{thm:EO1}, for any bounded connected regular open set $D$ whose closure has a connected complement, after a linear change of coordinates, we can choose a sequence $(D_j)_{j=0}^{\infty}$ of compact sets such that 
     \begin{itemize}
     \item $\Delta(0,c)\Subset D\subset D_0\subset\overline{\Delta(0, \frac{1}{3})}$, for some $c>0$,
    \item $D_j\subset\text{Int}(D_{j-1})$ for all $j\geq 1$,
    \item $\cap_{j=0}^{\infty}D_j=\overline{D}$.
\end{itemize}
Then we can choose a sequence $(B_{m,n})_{0\leq m, 1\leq n\leq 2^m}$ of open disks such that 
\begin{itemize}
    \item $B_{m,n}\subset\text{Int}(D_m)\setminus D_{m+1}$, for $0\leq m, 1\leq n\leq 2^m$
    \item $B_{m,n}\cap B_{m,n_1}=\emptyset$ for $n\neq n_1$,
    \item $(B_{m,n})'=\partial D$, that is, $\partial D$ is the Hausdorff limit of the sets $B_{m,n}$.
\end{itemize}

    We take $\Omega$, $\Delta_{m,n}$, $C_k$, $N_m$ and $h_{m,n}$ as in the proof of Theorem \ref{thm:EO1}. 

        Now, we define a linear map $h_{\delta}$ such that $h_{\delta}(D_0)\Subset \Delta(0,\frac{c}{2})\subset D^{-\delta}$. By induction as in Lemma \ref{lem:1}, we can obtain a sequence $(f_k)_{k\geq 1}$ of entire functions and a sequence $(x_n^j)_{n,j\geq 1}$ of points such that for all $k\geq 1$:
        \begin{enumerate}
         \item $|f_{k+1}-f_k|_{\overline{C_k}}\leq 2^{-k}$,
         \item $f^{N_{m}+2n-1}_k(\Omega_k)\subset B_{m,n}$, for  $0\leq m\leq k-1$ and $1\leq n\leq 2^{m}$.
         \item $f^{N_m}_k(\Omega_k)\subset \Delta_m$, for $1\leq m\leq k$,
         \item $f_k^j|_{\Omega_k}$ and $f_k^j|_{B_{m,n}}$ are univalent for all $1\leq j\leq N_{k-1}$, $0\leq m\leq k-1$ and $1\leq n\leq 2^m$.
         \item $f_k^j(x_n)=x_n^j$ for all $0\leq j\leq N_{k-1}$ and $1\leq n\leq k$, (here $x_n^0=x_n$).
         \item $f_k^{N_{n}}(x_n)=1$ for all $1\leq n\leq k$,
         \item $f_k(1)=1$ and $f_k'(1)=\frac{1}{2}$,
         \item $f_k(D^{-\delta})\Subset D^{-\delta}.$
     \end{enumerate}
     Indeed, we define compact sets $K_1$, $K_2$, $K_3$ and $K_{m,n}$ as in the proof of Lemma \ref{lem:1}. Let $K$ be a compact set such that $D^{-\delta}\subset K\subset D$.
     Then, those compact sets are disjoint and their complements are connected. Therefore, by Theorem A, given $\epsilon_1>0$, there exists an entire function $f_1$ such that
     \begin{itemize}
         \item $|f_1-h_j|_{K_j}\leq \epsilon_1/2$ for all $j=1,2,3$,
         \item $|f_1-h_{m,n}|_{K_{m,n}}\leq \epsilon_1/2$ for all $0\leq m, 1\leq n\leq 2^m$
         \item $|f_1-h_{\delta}|_K\leq \epsilon_1/2,$
         \item $f_1(x_1)=1$, $f_1(1)=1$ and $f'(1)=1/2.$
     \end{itemize}
     Then, the entire function $f_1$ satisfies the following properties:
     \begin{itemize}
   \item[(i).] $f_1^{N_0+1}(\Omega_1)=f_1(\Omega_1)\subset B_{0,1}$,
   \item[(ii).] 
   $f_1^{N_1}(\Omega_1)=f_1^2(\Omega_1)\subset\Delta_1$,
   \item[(iii).] $f_1^j|_{\Omega_1}$ and $f_1^j|_{B_{0,1}}$ are univalent for $1\leq j\leq N_0$,
   \item[(iv).] $f_1^{N_1}(x_1)=f_1(x_1)=1$,
   \item[(v).] $f_1(1)=1$ and $f_1'(1)=\frac{1}{2}$,
   \item [(vi).] $f_1(D^{-\delta})\Subset D^{-\delta}$.
   \end{itemize}
     
     We only check if (vi) holds, others can be done as in the proof of Lemma \ref{lem:1}.
     This indeed can be attained by $h_{\delta}(D_0)\Subset \Delta(0,\frac{c}{2})$ and
     \begin{align*}
         |f_{1}-h_{\delta}|_{D^{-\delta}}\leq \epsilon_1<\frac{c}{2}
     \end{align*} for small $\epsilon_1$. 
     The remaining part can be achieved by induction as in the proof of Lemma \ref{lem:1}.
     One can see that (8) holds inductively, as
      \begin{align*}
         |f_{k+1}-h_{\delta}|_{D^{-\delta}}\leq \sum_{j=1}^{\infty}\epsilon_j<\frac{c}{2}
     \end{align*} for small $\epsilon_j$ with $\epsilon_{j+1}<\epsilon_j$.
     Hence, we can construct an entire function $f$ such that $f(D^{-\delta})\subset D^{-\delta}$, which implies that $D^{-\delta}$ is an invariant Fatou set, while its boundary $\partial D$ is the limit of subsequences of $f^n$ on $\Omega$, and then a component of Julia set, thus there exists an invariant Fatou component $U_{\delta}$ such that $D^{-\delta}\subset U_{\delta}\subset D$.
   
\section{Proof of Theorem \ref{thm:Baker}}
     
         Without loss of generality, after a suitable affine transformation,  we can assume that $$U\subset\mathbb{H}:=\{z\in\C: \Re z<0\}.$$ For $j\geq 0$, define
         $$U_j:=\{z\in\C: \text{dist}(z, U)\leq \frac{1}{2^j}\}.$$

         Now, we choose a sequence $(B_{m,n})_{0\leq m, 1\leq n\leq 2^m}$ of open disks such that 
\begin{itemize}
    \item $B_{m,n}\subset\text{Int}(U_m)\setminus U_{m+1}$, for $0\leq m, 1\leq n\leq 2^m$
    \item $B_{m,n}\cap B_{m,n_1}=\emptyset$ for $n\neq n_1$,
    \item $(B_{m,n})'=\partial U$, that is, $\partial U$ is the Hausdorff limit of the sets $B_{m,n}$.
\end{itemize}
Now, we take $\Omega$, $(\Omega_n)_{n\geq0}$, $(x_n)_{n\geq 0}$, $(\Delta_{m,n})_{m\geq 1, 1\leq n\leq 2^m}$, $(\Delta_m)_{m\geq 1}$ and $(N_m)_{m\geq 0}$ as in the proof of Theorem \ref{thm:EO1}. Let $(C_k)_{k\geq 1}$ be the sequence of the left half-plane $$C_k:=\{z\in\C: \Re z<4k-2\}.$$ 
Clearly, $$\Delta_{k,n}\cup\Delta_k\subset C_{k+1}\setminus \overline{C_k}.$$
Let $h_{m,n}$ be a non-constant linear map such that 
$$h_{m,n}(B_{m,n})\Subset \Delta(b_{m,n},\frac{1}{2})\subset \Delta_{m,n}, 0\leq m, 1\leq n\leq 2^m-1$$ and $h_{m,2^m}$ be a non-constant linear map such that 
$$h_{m,2^m}(B_{m,2^m})\Subset\Delta(4(m+1),\frac{1}{2})\subset\Delta_{m+1}, m\geq 0.$$
Since $U$ is a nice domain, so is $U^{-\delta}$ for given $\delta>0$. By Lemma \ref{lem:nice}, there exists a non-constant holomorphic map $h:U^{-\delta}\to U^{-\delta}$ and an absorbing domain $W$ such that $\overline{h(U^{-\delta})}\subset U^{-\delta}$, $\overline{W}\subset U^{-\delta}$ and $h^n(\overline{W})=\overline{h^n(W)}\subset h^{n-1}(W)$ for all $n\geq 1$,  and $\text{dist}(h^n(W), h^{n-1}(\partial W))>d>0$ for all $n\geq 1$ and some constant $d$.

Repeat the same argument of Theorem \ref{thm:EO1}, we can also establish the following:
\begin{lemma}\label{lem:W}
There exists a sequence $(f_k)_{k\geq 1}$ of entire functions and a sequence $(x_n^j)_{n,j\geq }$ of points such that for all $k\geq 1$:
\begin{enumerate}
         \item $|f_{k+1}-f_k|_{\overline{C_k}}\leq 2^{-k}$,
         \item $f^{N_{m}+2n-1}_k(\Omega_k)\subset B_{m,n}$, for  $0\leq m\leq k-1$ and $1\leq n\leq 2^{m}$,
         \item $f^{N_m}_k(\Omega_k)\subset \Delta_m$, for $1\leq m\leq k$,
         \item $f_k^j|_{\Omega_k}$ and $f_k^j|_{B_{m,n}}$ are univalent for all $1\leq j\leq N_{k-1}$, $0\leq m\leq k-1$ and $1\leq n\leq 2^m$,
         \item $f_k^j(x_n)=x_n^j$ for all $0\leq j\leq N_{k-1}$ and $1\leq n\leq k$, (here $x_n^0=x_n$).
         \item $f_k^{N_{n}}(x_n)=1$ for all $1\leq n\leq k$,
         \item $f_k(1)=1$ and $f_k'(1)=\frac{1}{2}$,
         \item $\overline{f_k^n(W)}\subset f_{k}^{n-1}(W),$ for all $ n\geq 1$,  and $f_k(U^{-\delta})\subset U^{-\delta}$.
     \end{enumerate}
\end{lemma}
First, by Theorem B, it is not hard to construct an entire function $f_1$ satisfying:
   \begin{itemize}
   \item[(i).] $f_1^{N_0+1}(\Omega_1)=f_1(\Omega_1)\subset B_{0,1}$,
   \item[(ii).] 
   $f_1^{N_1}(\Omega_1)=f_1^2(\Omega_1)\subset\Delta_1$,
   \item[(iii).] $f_1^j|_{\Omega_1}$ and $f_1^j|_{B_{0,1}}$ are univalent for $1\leq j\leq N_0$,
   \item[(iv).] $f_1^{N_1}(x_1)=f_1(x_1)=1$,
   \item[(v).] $f_1(1)=1$ and $f_1'(1)=\frac{1}{2}$,
   \item [(vi).] $f_1^n(W)\subset f_1^{n-1}(W)$ and $f_1(U^{-\delta})\subset U^{-\delta}$.
   \end{itemize}
   Indeed, define closed sets $K_1$, $K_2$, $K_3$ and $K_{m,n}$ as in the proof of Lemma \ref{lem:1} and $K=\overline{U}$. It is clear that $$A:=\bigcup_{0\leq m, 1\leq n\leq 2^m, 1\leq j\leq 3}K_{m,n}\cup K_j\cup K$$ is a closed Weierstrass set. By Theorem B, for every $\epsilon>0$, there exists an entire function $f_1$ such that 
   \begin{itemize}
    \item $|f_1-h_j|_{K_j}\leq\frac{\epsilon_1}{2}$ for all $1\leq j\leq 3$,
    \item $|f_1-h_{m,n}|_{K_{m,n}}\leq\frac{\epsilon_1}{2}$ for all $0\leq m, 1\leq n\leq 2^m$,
    \item $|f_1-h|_K\leq \frac{\epsilon_1}{2}$,
    \item $f_1(x_1)=1$
    \item $f_1(1)=1$ and $f_1'(1)=\frac{1}{2}.$
\end{itemize}
Thus, (i)-(v) can be achieved as long as $\epsilon_1>0$ is small. We now check (vi) holds. For $z\in U^{-\delta}$, since $\overline{h(U)}\subset U^{-\delta}$, it follows that $\text{dist}(h(z),\C\setminus U^{-\delta})>\beta$ for some $\beta>0$, and then if $\epsilon_1$ is chosen to be smaller than $\beta$, we have 
\begin{align*}
    \text{dist}(f_1(z), \C\setminus U^{-\delta})&\geq\text{dist}(h(z),\C\setminus U^{-\delta})-|f_1(z)-h(z)|\\
    &\geq \beta-\epsilon_1>0.
\end{align*} 
Hence $f_1(U^{-\delta})\subset U^{-\delta}$. Similarly, we have $f_1(W)\subset W$.

Now, given a point $w\in W$, 
by Lemma \ref{lem:MRW}, for every $n\geq 1$, we have 
$$|f_1^j-h^j|_{\overline{W}}<\epsilon, 1\leq j\leq n$$
thus, we can choose $$\epsilon<\frac{d}{2}<\frac{1}{2}\text{dist}(h^{n}(W), h^{n-1}(\partial W)),$$
this implies that 
\begin{align*}
    \text{dist}(f_1^n(w), f_1^{n-1}(\partial W))&\geq \text{dist}(h^n(w), h^{n-1}(\partial W))-|f_1^{n-1}(w)-h^{n-1}(w)|\\
    &\quad\quad-\sup_{w'\in\partial W}|f_1^{n-1}(w')-h^{n-1}(w')|\\
    &\geq \text{dist}(h^n(w), h^{n-1}(\partial W))-2|f_1^{n-1}-h^{n-1}|_{\overline{W}}\\
    &>0.
\end{align*} Therefore, $\overline{f_1^n(W)}\subset f_1^{n-1}(W)$ for all $n\geq 1$.

Suppose that entire functions $f_1,\dots,f_k$ and points $(x_n^j)^{0\leq j\leq N_{n-1}}_{1\leq n\leq k+1}$ are constructed. Repeat the same proof of Lemma \ref{lem:1}, we can inductively obtain an entire function $f_{k+1}$ satisfying (1)-(7) in Lemma \ref{lem:W} as long as $$\sum_{j=1}^{\infty}\epsilon_j<\delta.$$ 
Now, we are going to check (8) holds. 
Indeed, one can get for $z\in U^{-\delta}$
$$\text{dist}(f_{k+1}(z), \C\setminus U^{-\delta})\geq \delta-\sum_{j=1}^{k+1}\epsilon_j>0.$$ 
Hence $f_{k+1}(U^{-\delta})\subset U^{-\delta}$ and also $f_{k+1}(W)\subset W$. Analog to the case of $f_1$, we can obtain that 
$$\overline{f_{k+1}^n(W)}\subset f_{k+1}^{n-1}(W)$$ for all $n\geq 1$.
This completes the proof of Lemma \ref{lem:W}.

Therefore, we can construct an entire function $f$ such that $f(U)\subset U$, which implies that $U$ is an invariant Fatou set. Observe that for every $n\geq 1$, our function $f$ satisfies $$|f^n-h^n|_{\overline{W}}<\sum_{k=0}^n\sum_{j=k+1}^{\infty}\epsilon_j<\sum_{k=1}^{\infty}k\epsilon_k.$$  Since $h^n|_W\to \infty$ as $n\to\infty$, it follows that $$f^n|_W\to\infty,\ n\to\infty$$ as long as $\sum_{k=1}^{\infty}k\epsilon_k<\infty$. 
Thus $U^{-\delta}$ is in an invariant Baker domain $V$. As $\partial U$ is a Julia component, it follows that $U^{-\delta}\subset V\subset U.$
 \section*{Acknowledgement}
 We would like to thank Yang Fei for pointing out that not all regular domains can be realized as a periodic Fatou component in the Workshop on Complex Dynamics and Complex Equations at TSIMF.
	\bibliographystyle{abbrv}
	\bibliography{con}
\end{document}